\documentclass{jgaa-art}

\usepackage{graphicx}

\usepackage{amsmath}
\usepackage{amssymb}
\newtheorem{thm}{Theorem}
\newtheorem{defn}{Definition}
\newtheorem{lem}{Lemma}
\usepackage{algorithm,algpseudocode}

\begin{document}

\doi{}
\Issue{0}{0}{0}{0}{0} 
\HeadingAuthor{Clancy et al.} 
\HeadingTitle{An effective crossing minimisation heuristic...} 

\title{An effective crossing minimisation heuristic based on star insertion}

\author[first]{Kieran Clancy}{kieran.clancy@flinders.edu.au}
\author[first,second]{Michael Haythorpe}{michael.haythorpe@flinders.edu.au}
\author[first]{Alex Newcombe}{alex.newcombe@flinders.edu.au}

\affiliation[first]{College of Science and Engineering,\\
Flinders University,\\
1284 South Road, Tonsley, Australia}

\affiliation[second]{Corresponding author}

\submitted{May 2018}%
\reviewed{November 2018}%
\revised{January 2019}%
\accepted{February 2019}%
\final{February 2019}%
\published{???}%
\type{Regular paper}%
\editor{G. Liotta}%

\maketitle

\begin{abstract}
We present a new heuristic method for minimising crossings in a graph. The method is based upon repeatedly solving the so-called {\em star insertion problem} in the setting where the combinatorial embedding is fixed, and has several desirable characteristics for practical use.  We introduce the method, discuss some aspects of algorithm design for our implementation, and provide some experimental results.  The results indicate that our method compares well to existing methods, and also that it is suitable for dense instances.
\end{abstract}

\Body

%
%
%
%
%


\section{Introduction}
\label{sec:intro}
The crossing number of a graph $G$, denoted by $cr(G)$, is defined as the minimum number of pairwise edge intersections of any drawing of $G$ in the plane.  Crossing numbers have been the source of many challenging problems since their introduction by Tur\'an in the 1950s \cite{tur}.  Determining the crossing number of a graph is NP-hard \cite{GareyJohnson} and remains NP-hard even for restricted versions of the problem, such as finding $cr(G)$ for a planar graph with the addition of a single edge \cite{cab}.  In recent years, an integer linear programming approach for finding $cr(G)$ has been developed in \cite{buch2}, \cite{chim} and \cite{chim:gut:mutz}.  This is successful for sparse graphs of moderate size, however ILPs and other exact methods are very limited when it comes to dense graphs.  For example, even $cr(K_{13})$ remains unknown \cite{pan}.  Therefore, heuristic methods are currently of interest and are the subject of our present work.  Interested readers are referred to \cite{buch} for a description of the current methods for computing crossing numbers, or to \cite{vrt} for a comprehensive bibliography of results up until 2014 related to crossing numbers.

This manuscript is organised as follows. For the remainder of this section, we give some necessary definitions. In Section \ref{sec:Related} we discuss some related approaches to crossing minimisation, most notably including the {\em planarisation method} and its variants, which are regarded as the current best practical methods for minimising crossings \cite{chim4}.  We propose our new heuristic method in Section \ref{sec:newmethod} and discuss some implementation and design aspects in Section \ref{sec:designmeth}.  In Section \ref{sec:runtime} the runtime is shown to be $O((k+n)m)$ time per iteration, where $k$ is the number of crossings in a current drawing of $G$.  Lastly, in Section \ref{sec:experiments}, we benchmark the new heuristic on several collections of graphs which have been used in the past for testing other crossing minimisation algorithms.  A fully featured implementation of our heuristic, which we call QuickCross, is available at http://fhcp.edu.au/quickcross in both C and \textsc{MATLAB} format.

\subsection{Definitions}
For an undirected graph $G$ with vertex set $V(G)$ and edge set $E(G)$, let $n=|V(G)|$ and $m=|E(G)|$.  Let $d(v)$ be the degree of vertex $v$ and $\Delta(G)$ denote the maximum degree of any vertex in $G$.  For a vertex $v \in V(G)$, let $N_G(v)$ denote the neighbourhood of $v$; that is, the set of vertices such that there exist edges in $G$ connecting $v$ to those vertices.

An {\em embedding} of a graph $G$ onto a surface $\Sigma$ (a compact, connected 2-manifold) is a representation of $G$ onto $\Sigma$ such that vertices are distinct points on $\Sigma$ and each edge $e$ is a simple arc on $\Sigma$ connecting the points associated with the end vertices of $e$.  The embedding must also satisfy: An arc of edge $e$ does not include any points associated with vertices other than the end vertices of $e$, and two arcs never intersect at a point which is interior to either of the arcs.  Two embeddings are equivalent if there is a homeomorphism of $\Sigma$ which transforms one into the other.  The equivalence class of all such embeddings is a {\em topological embedding} of $G$.  In this present work, we are only concerned with embeddings in which $\Sigma$ is the surface of a sphere and this is assumed to be the case for the rest of this paper.  A topological embedding of $G$ onto the sphere uniquely defines a cyclic ordering of the edges incident to each vertex of $G$ and the collection of these cyclic orderings is a {\em combinatorial embedding} for $G$.  A combinatorial embedding $\Pi$ defines a set of cycles in $G$ which bound the faces of any embedding belonging to the associated topological embedding, and so we may talk about the set of `faces' of $\Pi$.  Similarly, $\Pi$ defines a dual graph $\Pi^*$ which is isomorphic to the dual graph of any embedding belonging to the associated topological embedding.  Note that the edges $e_1,e_2,\dots,e_m$ of $G$ are in one-to-one correspondence with edges $e^*_1,e^*_2,\dots,e^*_m$ of the dual graph $\Pi^*$.

A {\em drawing} $D$ is a representation of a graph $G$ onto the plane with similar conditions to an embedding.  Vertices are represented as distinct points and each edge $e$ is represented by a simple arc between the points associated with the end vertices of $e$.  The drawing must also satisfy: An arc of edge $e$ does not include any points associated with vertices other than the end vertices of $e$, and any intersection between the interiors of arcs involves at most two arcs.  Given a drawing $D$ of $G$, the intersections which occur in the interiors of arcs are the {\em crossings} of the drawing and the number of crossings is denoted by $cr_D(G)$.  The {\em crossing number} of a graph is denoted by $cr(G)$ and is the minimum number of crossings over all possible drawings of $G$.  If $cr_D(G)=0$, then $G$ is planar, and we say that $D$ is a {\em planar drawing} of $G$.  A {\em planarisation} of a drawing is a planar drawing of the planar graph obtained by replacing crossings of the initial drawing with dummy vertices of degree 4.  Hence the graph corresponding to the planarisation of $D$ has $n+cr_D(G)$ vertices and $m + 2cr_D(G)$ edges.

In what follows, when no confusion is possible, we shall refer to the arcs of a drawing or embedding as `edges' of the drawing (or embedding) and the points associated with vertices as `vertices' of the drawing (or embedding).

Throughout this paper, we will often consider the situation where we have a combinatorial embedding $\Pi$ of a graph, and then need to add an edge $e$ to the graph and obtain an updated combinatorial embedding. In such cases, we will say that we are {\em inserting} an edge $e$ into $\Pi$, as follows. Suppose that $e = (v_1,v_2)$, where $v_1,v_2 \in V(G)$, and let $\Gamma$ be an embedding which realises the cyclic orderings in $\Pi$.  A simple arc connecting $v_1$ and $v_2$ may be added to $\Gamma$, such that the interior of the arc intersects only with the interiors of some (possibly empty) ordered set of edges $\{e_1,e_2,\dots,e_k\}$ already present in $\Gamma$.  Clearly, for any embedding which realises the cyclic orderings in $\Pi$, such an arc can be found which intersects exactly the same set of edges $\{e_1,e_2,\dots,e_k\}$.  We also refer to these intersections as `crossings'.
	
In the following discussion, we will be defining the \textit{star insertion problem} (SIP).  Prior to that, we first consider the \textit{edge insertion problem} (EIP) studied in \cite{gut2}, which has two variations depending on the definition of optimality used; the fixed embedding variation and the variable embedding variation:
\begin{defn}{(EIP -- fixed embedding)}
Given a combinatorial embedding $\Pi$ of a graph $G$ and a pair of vertices $v_1,v_2 \in V(G)$, insert the edge $e=(v_1,v_2)$ into $\Pi$ in such a way that the number of crossings is minimised.
\end{defn}
\begin{defn}{(EIP -- variable embedding)}
Given a planar graph $G$ and a pair of vertices $v_1,v_2 \in V(G)$, find a combinatorial embedding $\Pi$ of $G$ such that inserting the edge $e=(v_1,v_2)$ into $\Pi$ so as to minimise the crossings results in the minimal number of crossings among all embeddings of $G$.
\end{defn}
The fixed embedding problem can be solved in $O(n)$ time by finding a shortest path in a modified dual graph, and this is explained in detail in \cite{gut}.  In \cite{gut2} it is shown that the variable embedding problem can also be solved in $O(n)$ time by taking advantage of the properties of maximal tri-connected components and SPQR trees.  Note that solving the edge insertion problem is different from computing the crossing number of $G + e$ (the graph $G$ with the addition of edge $e$).  However, it is shown in \cite{hlin} that the number of crossings introduced in a solution to the variable embedding version approximates $cr(G + e)$ to within some factor and the best possible factor is proved in \cite{cab2} to be $\lfloor \Delta (G)/2 \rfloor$.

A natural extension to the above is the {\em star insertion problem} (SIP) where instead of a single edge, the object to be added to $G$ is a vertex $v$ along with a set of incident edges of $v$ (collectively, a star).  As before, there are fixed embedding and variable embedding versions of SIP:

\begin{defn}{(SIP -- fixed embedding)}
Given a combinatorial embedding $\Pi$ of a graph $G$ and a vertex $v \not\in V(G)$ (along with a set of incident edges whose other endpoints are all in $V(G)$), insert $v$ along with its incident edges into $\Pi$ in such a way that the number of crossings is minimised.
\end{defn}
\begin{defn}{(SIP -- variable embedding)}
Given a planar graph $G$ and a vertex $v \not\in V(G)$ (along with a set of incident edges whose other endpoints are all in $V(G)$), find a combinatorial embedding $\Pi$ of $G$ such that inserting $v$ along with its incident edges into $\Pi$ so as to minimise the crossings results in the minimal number of crossings among all embeddings of $G$.
\end{defn}

The fixed embedding version can be solved in $O(d(v) \cdot n)$ time using a method similar to the single edge insertion version \cite{chim3}. We will make use of this approach during our heuristic, and we briefly outline our implementation of this method in Section \ref{sec:newmethod}. The complexity of the variable embedding version was in question for a short time but was resolved by Chimani et al \cite{chim3} who showed it to be solvable in polynomial time by a method which is briefly outlined in Section \ref{sec:Related}.  Again, the number of crossings introduced in a solution to the variable embedding version is shown in Chimani, Hlin\v en\'y and Mutzel \cite{chim2} to approximate the crossing number of the graph $G+v$ to within a factor of $d(v) \lfloor \Delta(G)/2 \rfloor$.

\section{Related work}\label{sec:Related}
Crossing minimisation has been considered in a number of contexts. For example, in the field of automated graph drawing, heuristics have been developed to construct drawings of graphs or networks with desirable characteristics, which often includes a low number of crossings. Approaches including force-directed drawing algorithms \cite{Eades,FM3,Kawai} and genetic algorithms \cite{Branke,Eloranta,barreto} have been developed for this purpose. When crossing minimisation is the sole aim, arguably the most successful heuristics to date have been based on edge insertion procedures.

\subsection{Planarisation method}

The planarisation method, a highly effective crossing minimisation heuristic, is based upon repeatedly solving the edge insertion problem. In particular, the planarisation method involves attempting to solve two separate problems:

\begin{enumerate}
\item Compute a planar subgraph $G_p$ of $G$ - ideally a maximum planar subgraph.
\item Iteratively re-insert the remaining edges of $G$ into a combinatorial embedding of $G_p$ while striving to keep number of crossings as small as possible.
\end{enumerate}

Computing a maximum planar subgraph is NP-hard \cite{liu}, so instead a locally maximal planar subgraph is usually used for step 1, which can be computed in $O(n+m)$ time \cite{Djidjev}. To achieve step 2, given a planar subgraph of $G$, EIP (in the fixed or variable embedding) is solved for one of the missing edges. Then any introduced crossings are replaced by degree 4 dummy vertices to obtain a new planar graph, and EIP is solved again for another missing edge, and so on until an embedding of the full graph is obtained.

The planarisation method was first described in the context of EIP-fixed by Batini et al \cite{batini}.  Later, in Gutwenger \cite{gut}, the method was rigorously developed for EIP-variable, along with an implementation and experimental results which were also reported in Gutwenger and Mutzel \cite{gut:mutz}.  In most cases, the method based on EIP-variable provided superior solutions for the tested graphs. However, it was observed that the EIP-variable method often suffered in runtime in comparison to EIP-fixed implementations, due to the many SPQR trees which need computed (a new SPQR tree for every edge inserted).  Later, in Chimani and Gutwenger \cite{chim4}, implementations were also reported on which focused on improving the post processing schemes that can be utilised when running these methods and again improved results were obtained from those previously reported.

A related approach to the planarisation method is to solve the multiple edge insertion problem (MEI), which involves inserting several edges simultaneously into a planar graph.  Let $F$ be the set of edges being inserted into some planar graph $G$.  For general $F$, solving MEI to optimality is NP-Hard \cite{Zieg}, and approximation algorithms have been developed in \cite{chuzhoy} and \cite{chim5}.  An approximate solution to MEI is known to approximate the crossing number of the graph $G+F$ \cite{chim2} and so for graphs of bounded degree and bounded $|F|$, the algorithm in \cite{chuzhoy} constitutes a multiplicative factor approximation algorithm for $cr(G+F)$ and the algorithm in \cite{chim5} constitutes an additive factor approximation algorithm for $cr(G+F)$.  Among implementations based on MEI, only the algorithm of Chimani \cite{chim5} has been experimentally reported on. In particular, it was considered in Chimani and Gutwenger \cite{chim4}, which is the most recent analysis on the practical usage of various crossing minimisation heuristics. Chimani and Gutwenger \cite{chim4} claim that the MEI implementation from \cite{chim5} achieves roughly comparable solution quality to the best iterative EIP-variable method, with the benefit of significantly reduced runtimes. If runtimes are disregarded, the iterative EIP-variable method with the addition of a significant post processing step usually produced the best solutions, however overall (in terms of both solution quality and runtime) Chimani and Gutwenger \cite{chim4} advocate that the MEI implementation from \cite{chim5} was the best heuristic for practical use.

Recently, in \cite{Radermacher}, another variation of EIP was investigated.  Given a combinatorial embedding $\Pi$ of $G$ and $v_1,v_2 \in V(G)$, the task is to find a straight line drawing of $G$ which realises the cyclic orderings of $\Pi$, and such that a straight line between $(v_1,v_2)$ can be added with minimal number of crossings among all straight line drawings of $G$ which realise the cyclic orderings of $\Pi$.  This problem is known as geometric edge insertion and for the case $\Delta(G) \leq 5$, can be solved in linear time by the algorithm in \cite{Radermacher}.

\subsection{Methods based on star-insertion}
\label{sec:planmethod}
The approach to solving SIP-variable, described in Chimani et al \cite{chim3}, can be summarised as follows for a given graph $G$ and vertex $v$ to be inserted.

\begin{enumerate}
\item Compute an SPQR tree $T$ of $G$, and consider a face $f$ in one of the skeleton graphs of $T$ ($f$ belongs to a set of `interesting' faces).
\item Solve a dynamic program whose solution advises the best combinatorial embedding which admits the minimal number of crossings when inserting $v$ into $f$.
\item Repeat the above for all `interesting' faces and select the solution which results in the fewest crossings.
\end{enumerate}

Although the runtime of the algorithm provided in \cite{chim3} is polynomial, it is considerably higher than solving EIP-variable, and experimental results have yet to be reported on. Nonetheless, a heuristic analogous to the planarisation method, but using star insertion rather than edge insertion, could be proposed. Indeed, in Chimani and Gutwenger \cite{chim4}, it is asked whether a heuristic based on star insertion could compare to the proven practical performance of the heuristic methods based on edge insertion. This present work seeks to answer this question, at least for SIP-fixed, but the approach we advocate is different in character to the planarisation method.

In particular, the approach that we advocate is to iteratively obtain improved drawings of a graph in the following way. For a given drawing $D$ of a graph $G$, we attempt to find a vertex $v$ in $G$ satisfying the following: if we remove $v$, and then reintroduce $v$ by solving the star insertion problem in a corresponding (fixed) combinatorial embedding, a drawing $D_2$ can be obtained such that $cr_{D_2}(G) < cr_D(G)$. If there are no vertices in the graph for which this is possible, we say that the drawing $D$ is {\em locally crossing-optimal}. In what follows, we will prove the following.

\begin{thm}Let $G$ be a graph containing $n$ vertices and $m$ edges, and let $D$ be a drawing of $G$ which contains $k$ crossings. There exists an algorithm that finds a locally crossing-optimal drawing $D^*$ of $G$ in $O((k+n)km)$ time.\end{thm}

It is our contention that the number of crossings in such a $D^*$ found by our algorithm is, typically, close to the crossing number of $G$. We provide experimental results justifying this assertion in Section \ref{sec:experiments}.

\section{Proposed heuristic method}
\label{sec:newmethod}
While the philosophy of the planarisation method is to start with a planar subgraph and increase the number of crossings at each iteration as the full graph is rebuilt, our approach works in the opposite direction; we start with a combinatorial embedding corresponding to a, presumably suboptimal, drawing of the full graph and at each iteration we attempt to find a combinatorial embedding corresponding to a drawing with fewer crossings.  Unlike the planarisation method, the heuristic we propose does not require a planar subgraph to be computed.  Instead it relies upon iteratively solving the star insertion problem in a combinatorial embedding which corresponds to the current (non-planar) drawing of $G$.  With the intention of keeping the new heuristic highly practical, each iteration is performed on a fixed combinatorial embedding; this is discussed further in Section \ref{sec:designmeth}.

Let $D$ be some drawing of $G$ and let $D'$ be its planarisation.  Then $D'$ can be mapped to an embedding on the sphere, and this realises a particular combinatorial embedding.  In this sense, we say that the combinatorial embedding `corresponds' to the drawing $D$.  Note that given such a combinatorial embedding, a drawing which is equivalent to $D$ can be retrieved by using any planar graph drawing techniques, such as \cite{Fray} or \cite{Schny}.

Let $D$ be a drawing of $G$ and let $\Pi$ be a combinatorial embedding corresponding to $D$.  Consider deleting from $G$ a vertex $v$ and its set of incident edges; it is clear that a subdrawing $D-v$ can be easily obtained from $D$.  Then a combinatorial embedding corresponding to the subdrawing $D-v$ can be computed by repeatedly merging faces of $\Pi$ which share an edge associated with one of the deleted edges.  We shall call this the {\em reduced combinatorial embedding} corresponding to subdrawing $D-v$ and denote it as $\Pi-v$.

We define a star insertion into a combinatorial embedding $\Pi$ by utilising definitons similar to those in \cite{gut2}.  Let $\Pi$ be a combinatorial embedding of $G$, let $f$ be a face of $\Pi$ and let $v$ be a vertex of $G$.  Then $e_1,e_2,\dots,e_j$ is an {\em insertion path} for $v$ and $f$ if either $j=0$ and $v$ is on the boundary of $f$, or the following conditions are satisfied:
\begin{enumerate}
\item $e_1,e_2,\dots,e_j \in E(G)$.
\item There is a face of $\Pi$ with both $e_j$ and $v$ on its boundary.
\item $e_1$ is on the boundary of $f$.
\item $e_1^*,e_2^*,\dots,e_j^*$ is a path in the dual graph $\Pi^*$.
\end{enumerate}

Given an insertion path, an edge can be inserted into $\Pi$ starting from an arbitrary point in face $f$ (consider this a `dummy vertex' for the moment) and ending at vertex $v$ in such a way that it crosses precisely the edges $e_1,e_2,\dots,e_j$.

Then, suppose we have a collection of insertion paths $p_1,p_2,\dots,p_\ell$ whose associated end vertices are $v_1,v_2,\dots,v_\ell$. If they can all be inserted into $\Pi$ in the above fashion, such that they are pairwise non-crossing, then we say that they collectively constitute a {\em star insertion path}.  By inserting a dummy vertex $z$ into face $f$ and attaching the beginnings of each insertion path to $z$, the star comprising of $z$ and the edges $\{(z,v_i) \mid i=1,2,\dots,\ell\}$ can be inserted into $\Pi$ in such a way that they cross precisely the edges in $p_1,p_2,\dots,p_\ell$.  For a fixed face $f$, and a fixed set of end vertices $S=\{v_1,v_2,\dots,v_\ell\}$, we say that a star insertion path which crosses the fewest edges with respect to all possible star insertion paths into $f$ with the end vertices $S$, is a {\em crossing minimal star insertion path} for $f$ and $S$.

At each iteration we begin with a combinatorial embedding $\Pi$ corresponding to a some drawing of $G$.  The processes within an iteration are summarised in the following procedure, for a given vertex $v \in V(G)$:

\subsection*{\textbf{Procedure 1:}}
\begin{itemize}
\item [P1:] Compute the reduced  combinatorial embedding $\Pi-v$.

\item [P2:] Intelligently (see below) choose a face $f$ of $\Pi-v$.  Compute the number of crossings resulting from a crossing minimal star insertion path into face $f$ for the star comprising of $v$ and its incident edges.

\item [P3:] If the total number of crossings has reduced, then insert the star comprised of $v$ and its incident edges into $f$ according to a crossing minimal star insertion path.

\item [P4:] Replace each introduced crossing with a dummy vertex of degree 4, and obtain a new combinatorial embedding.  Call this new embedding $\Pi$ and begin the next iteration.
\end{itemize}
Note that Step P2 is equivalent to solving the fixed embedding star insertion problem for the vertex $v$ (and its incident edges) in $\Pi-v$. To achieve this, we use the algorithm described in Chimani et al \cite{chim3} on page 376. Since this is an important step in our heuristic, we include its description here. We begin by utilising a simple merging procedure in the dual graph of $\Pi-v$.  For each vertex $w \in N_{G}(v)$, we perform the following steps:

\begin{enumerate}
\item Contract the cycle in the dual graph that is formed by dual vertices of those faces that are incident to $w$, into a single vertex $d_w$ (see Figure \ref{fig1} for an example.)  Remove any resulting multi-edges.
\item Find shortest paths in the dual graph with $d_w$ as the source.
\item Store the distances to each dual vertex, and for those dual vertices that were contracted in step 1, set their distance to zero.
\item Discard changes to the dual graph so that the above steps can be repeated with a different neighbor of $v$.
\end{enumerate}

After the above procedure, the dual vertex possessing the minimum sum of distances (over all $w \in N_G(v)$) corresponds to the optimal face for the new placement of $v$, and the optimal insertion paths can be determined from the shortest path trees.

\begin{figure}[H]
\begin{centering}
\includegraphics[width=0.6\linewidth]{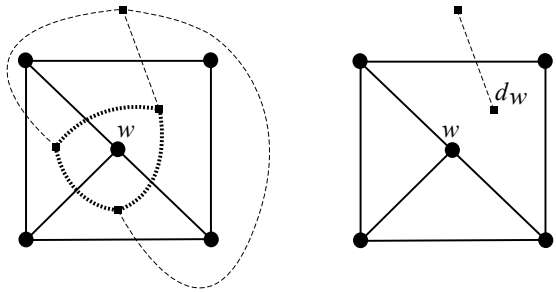}
\caption{The contraction operation in the dual graph for the vertex $w \in N_G(v)$.  Dual edges are dashed and dual vertices are squares.  The thick dashed edges which form a cycle around $w$ are contracted and multi-edges are discarded.  The resulting `merged' dual vertex is $d_w$.  In this particular case, the resulting dual graph is just a $K_2$. \label{fig1}}
\end{centering}
\end{figure}

\section{Design methodology}
\label{sec:designmeth}
In this section we outline some of the design choices and data structures of the highly practical implementation which is used for the experiments described in Section \ref{sec:experiments}.

\subsection{Initial embedding schemes}
\label{sec:init}
Since we focus on a fixed embedding at each iteration, the initial combinatorial embedding of $G$ obviously plays a large role in the performance of the heuristic.  Any drawing method can be used to compute an initial embedding, and we discuss below just three possiblities.  The first method produces an embedding quickly, however the initial number of crossings can be as large as ${n\choose4}$.  The second method is slower to compute but the initial number of crossings is usually much smaller for the case of sparse graphs.  The third method is an implementation of a force-directed graph drawing algorithm. We will refer to these three initial embedding schemes as {\em circle}, {\em planar} and {\em spring}, respectively.

The ``{\em circle}" initial embedding scheme, produces an embedding using the following procedure. We first assign each vertex a coordinate on the unit circle. Specifically, we place each vertex $i = 1, ... , n$ at coordinate $\left(\cos(\frac{2i\pi}{n}),\sin(\frac{2i\pi}{n})\right)$. Then, the edges are drawn as straight lines, and the crossings can be easily computed. An upper bound on the number of crossings for a drawing obtained by this method can be seen by following a simple counting argument:

\begin{lem}
The maximum number of crossings in a drawing obtained by the circle embedding scheme is ${n\choose4} = \frac{1}{24}(n^4 - 6n^3 + 11n^2 -6n)$.
\end{lem}

\begin{proof}
The maximum number of crossings is attained by the complete graph $K_n$.  In $K_n$, label the vertices from 1 to $n$ in a clockwise fashion, then any set of 4 vertices $\{ a,b,c,d \}$, where $a<b<c<d$, corresponds to exactly one crossing involving the edges $(a,c)$ and $(b,d)$.  Thus the total number of crossings is ${n\choose4}$. 
\end{proof}

The second initial embedding scheme, which we call ``{\em planar}", utilises a sequence of solutions to the star insertion problem. This idea has been considered as a heuristic for crossing minimisation in its own right (e.g. see Chimani et al \cite{chim3}), and involves constructing an embedding in a way which is similar to the planarisation method.  We begin by finding any chordless cycle of $G$ (if none exist then $G$ is acyclic and $cr(G) = 0$) along with an embedding $\Pi$ of this cycle, then iteratively perform the following:

\begin{enumerate}
\item Find a vertex $v \in V(G)$ which is not yet in $\Pi$, and such that there exists at least one edge in $E(G)$ which connects $v$ to a vertex already present in $\Pi$.  Denote by $F$ the set of all edges between $v$ and any vertices already present in $\Pi$.
\item Find a face $f$ of $\Pi$ such that a crossing minimal star insertion path, into $f$, of the star comprising of $v$ and the edges in $F$, introduces the least number of crossings among all faces of $\Pi$.
\item Insert, into $f$, the star comprising of $v$ and the edges in $F$ according to a crossing minimal star insertion path.
\item Replace each introduced crossing with a dummy vertex of degree 4 to obtain a planar graph, and compute a new combinatorial embedding.  Call this new embedding $\Pi$ and begin the next iteration.
\end{enumerate}

At each step of the procedure we are building upon the embedding, one vertex at a time, until we have an embedding corresponding to some drawing of the full graph $G$. As will be demonstrated in Section \ref{sec:experiments}, this method, although still computationally efficient, is in practice slower than the circle embedding, particularly for dense instances.  However, in Section \ref{sec:experiments} it will also be seen that this method tends to result in many fewer crossings, and hence substantial processing time is saved in the subsequent iterations of the main heuristic.  For this reason, this is the default embedding choice in our implementation of the heuristic.

The third initial embedding scheme, which we call ``{\em spring}", comes from the area of force-directed graph drawing. In \cite{Kawai} Kamada and Kawai describe a method for drawing a graph which minimises the energy of a spring model representation of the graph.  The resulting number of edge crossings is not taken into consideration in the spring model, however, especially for the case of sparse graphs, it will be demonstrated in Section \ref{sec:experiments} that the resulting drawings often provide a initial embedding with relatively few crossings. Of course, there are other force-directed graph drawing algorithms which could be used (e.g. see \cite{Eades,FM3}) and we make no claim here that \cite{Kawai} is the best for use in our heuristic.

It should be noted that, technically, any combinatorial embedding corresponding to a valid drawing of $G$ can serve as an initial embedding.  Indeed, our heuristic could be applied as a post-processing step of the planarisation method, or any other similar heuristic which results in a valid drawing.  To accommodate this, we have included in our implementation an option for user to specify their own initial combinatorial embedding, or to provide vertex coordinates for a straight-line drawing obtained from any drawing routine.

\subsection{Minimisation schemes}
\label{sec:minim}
There is a certain amount of freedom in the choice of how the heuristic descends towards its solution and we call these choices minimisation schemes. In particular, we discuss three possible minimisation schemes here.

The first minimisation scheme, which we call ``{\em first}" works as follows. We consider vertices one at a time, in the order of their labels. In the first iteration, the first vertex considered is the one with the earliest label, and in subsequent iterations the first vertex considered is the one that follows the vertex that was re-inserted in the previous iteration. As soon as a vertex is found which can be re-inserted in such a way that the number of crossings is reduced, we fix this improved position and begin the next iteration.

The second minimisation scheme, which we call ``{\em best}", works as follows. We consider each of the vertices, and determine which should be re-inserted so as to gain the greatest reductions in crossings. Then, we fix the improved position of that vertex and begin the next iteration.

The third minimisation scheme, which we call ``{\em biggest face}", comes from an observation made during experimentation; re-inserting a vertex $v$ into the face of $\Pi - v$ with the most edges (the `biggest face') often provides an improvement.  Intuitively this makes sense as the biggest face is `close' to a relatively large number of vertices.  This scheme allows for a significant speed increase during the early iterations because we may assume that the vertex can be placed in the biggest face and then find the shortest paths only once, using the dual vertex corresponding to the biggest face as the source, as opposed to the other schemes which require shortest paths to be computed up to $\Delta (G)$ times.  As will be shown in Section \ref{sec:runtime}, computing the shortest paths is the most time-consuming process in our heuristic and hence for dense graphs, where $\Delta (G) = \Theta(n)$, we gain a significant speed increase.  If the biggest face does not provide an improvement, other faces can then be checked according to one of the other minimisation schemes.  In our implementation, if it happens that the biggest face does not provide an improvement for a certain number (specified by the user) of consecutive iterations, we stop checking the biggest face first and instead continue with a different minimisation scheme from that point forward.

\subsection{Efficiently handling the dual graph}
\label{sec:dualgraph}
In each iteration, and for each vertex considered, the steps of the heuristic require the dual graph of the current embedding minus one vertex.  It is possible that we may need to consider many or even all of the vertices, particularly if we use the ``best" minimisation scheme. It is obviously undesirable to reconstruct this dual graph for every vertex, and so we use a simple updating procedure to avoid this.  We compute the dual graph once per iteration, with all vertices present. Then, each time a vertex (along with its incident edges) is deleted from $G$, the result in the embedding is that some pairs of faces (on either side of the planarised edges being deleted) are merged into single faces.  In the computed dual graph, this corresponds to contracting the dual edge connecting the two faces on either side of each of these planarised edges (see Figure \ref{fig2}).  Recall that each edge of the embedding corresponds precisely to an edge of the dual graph.  We keep these edge indices consistent in our implementation to help simplify the above process.

\begin{figure}[H]
\begin{centering}
\includegraphics[width=0.95\linewidth]{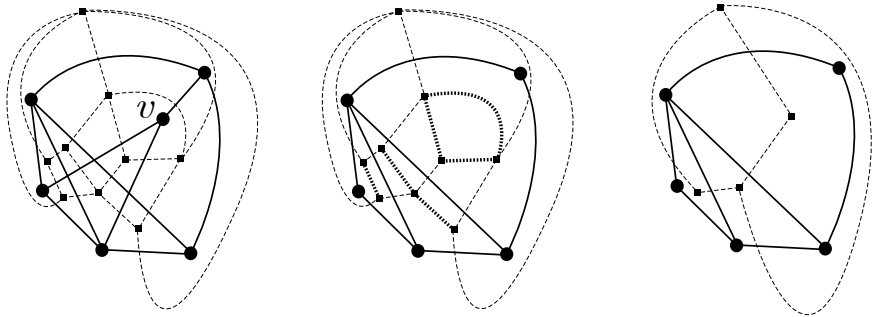}
\caption{Edge deletions and the corresponding edge contractions in the dual graph.  Dual edges are dashed and dual vertices are squares.  Vertex $v$ is to be deleted from $G$.  The middle picture has $v$ deleted and the thick dashed edges are the dual edges which are to be contracted.  The right picture is the result after multi-edges have been discarded. \label{fig2}}
\end{centering}
\end{figure}

\subsection{Pre- and post-processing schemes}
\label{sec:prepost}
Pre-processing schemes for crossing number heuristics are well understood and are reported in \cite{gut:mutz} and \cite{chim4}.  We briefly outline the usual pre-processing schemes.  The crossing number of a disconnected graph is the sum of crossing numbers of each of its connected components.  Similarly, the crossing number of a 1-connected graph is the sum of crossing numbers of its maximal bi-connected components. Therefore, we can decompose any input graph into its biconnected components and handle them individually. One benefit is that this allows us to assume that any graph submitted to our heuristic is biconnected, and hence we can assume that any graph with one vertex removed is connected.

Due to the desire to compare our heuristic to current methods, we have not currently implemented any post-processing schemes.  However several effective post-processing strategies are discussed in \cite{chim4} and could be appended to our heuristic if desired.

\subsection{Data structures}
\label{sec:datastruct}
To store a combinatorial embedding $\Pi$, a list structure containing the following information is utilised:  For each edge $e = (u,v)$, this list stores $u$ and $v$ along with four indices; the edge index of the edge immediately clockwise from $e$ around vertex $u$, the edge index of the edge immediately anti-clockwise from $e$ around vertex $u$, and then likewise for vertex $v$.  An example of this list can be seen in Figure \ref{fig3}.

The following list structures allow for the efficient modifcations of the embedding at each iteration.  The {\em crossing order} of an edge $e = (u,v)$ where $u<v$ is a list of the edges which currently cross $e$ in the order starting from the closest crossing to $u$.  Along with the crossing order list, there is the {\em crossing orientation list}.  The crossing orientation is essentially the cyclic order of edges around a dummy vertex in the embedding.  Suppose that within the crossing order entries of edge $e_1=(u_1,v_1)$, we have the entry $e_2=(u_2,v_2)$ where $u_1<v_1$ and $u_2<v_2$.  Then the corresponding crossing orientation entry is stored as 1 to indicate that the order of the edges when traversing clockwise around the dummy vertex have the end-vertices $u_1,u_2,v_1,v_2$, or -1 to indicate that the order is $u_1,v_2,v_1,u_2$.  Note that these are the only two possible orders (see Figure \ref{fig4} for an example).

One difficulty arising from the combination of using these data structures and working in a fixed embedding scheme is that a pair of edges may cross each other more than once.  Of course, it is known that in an optimal embedding this is never the case.  However, it can arise during an intermediate step of the heuristic.  If this happens, the crossing order list has no information about which entry corresponds to which crossing.  To avoid this confusion, if edges $e_1$ and $e_2$ cross each other more than once, then $e_1$ is subdivided into a chain of edges such that none of the resulting edges cross $e_2$ more than once.  A check is then performed in future iterations to see if the set of edges resulting from an earlier subdivision still cross any edge more than once.  If not, those subdivisions are removed and the edges are merged back into a single edge.  It is a simple excercise to show that by the time the heuristic concludes, all previous subdivisions have been reverted.  Note that, in practice, these subdivisions are a rare occurance.

\begin{figure}
\begin{centering}
\includegraphics[width=0.8\linewidth]{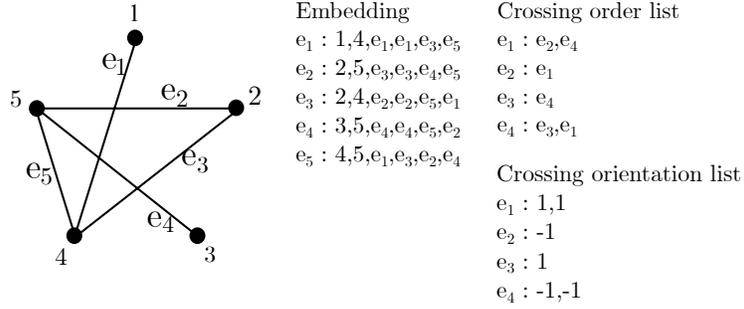}
\caption{A drawing associated with a combinatorial embedding and an example of the data structures utilised to store the embedding.  For an edge $e = (u,v)$ where $u<v$, the entries in the embedding list are: $u$, $v$, the edge clockwise from $e$ around $u$, the edge anti-clockwise from $e$ around $u$, the edge clockwise from $e$ around $v$, the edge anti-clockwise from $e$ around $v$. \label{fig3}}
\end{centering}
\end{figure}

\begin{figure}
\begin{centering}
\includegraphics[width=0.67\linewidth]{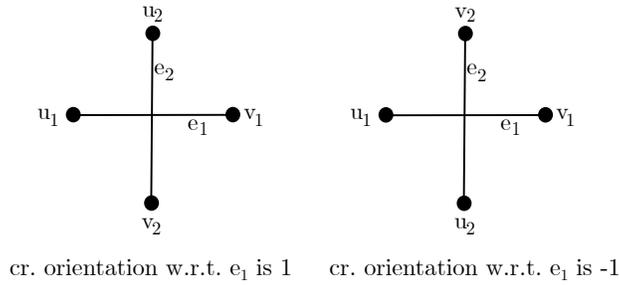}
\caption{If $e_1=(u_1,v_1)$ crosses $e_2=(u_2,v_2)$ where $u_1<v_1$ and $u_2<v_2$, the two possibilites for the crossing orientation are displayed. \label{fig4}}
\end{centering}
\end{figure}

\section{Runtime and implementation}
\label{sec:runtime}
We now discuss the runtime of each of the procedures and show that the iterations of the heuristic run in $O((k+n)m)$ time where $k$ is the number of crossings in the drawing associated with the current embedding of $G$.  Pseudocode for the main loop and two subroutines are displayed in Algorithms 1-3.  The code demonstrates the first minimisation scheme discussed in Section \ref{sec:minim} where an improvement is taken as soon as it is found.  There is a level of abstraction left in the pseudocode due to the numerous ways that one could perform the required operations; highly efficient C and MATLAB implementations of the heuristic are available at http://fhcp.edu.au/quickcross.  In the discussion below we refer to the pseudocode and summarise the methods used in our implementation.
%

\subsection{Implementation}

First we discuss the operations involved in Algorithms 1-3.  During the \textsc{main loop} procedure, we remove vertex $v$ and its incident edges, which possibly reduces the current number of crossings.  Then, after identifying the best possible new placement for $v$ using the \textsc{sip}$(G,\Pi^*,v)$ procedure, we have a new number of crossings for $v$'s potential placement and this number is \textit{new cr}.  Hence if  $\textit{new cr} < \textit{current cr}$ then we have found a drawing with fewer crossings.  Once an improvment has been found, $\Pi$ is updated to reflect the new placement and this involves updating each of the data structures discussed in Section \ref{sec:datastruct}.  If an edge is drawn such that it crosses some other edge multiple times, then we subdivide that edge to avoid confusion in the data structures as also discussed in Section \ref{sec:datastruct}.  Similarly, if a set of edges resulting from an earlier subdivision no longer crosses any edge multiple times, then the previous subdivisions are reverted.

From the current combinatorial embedding $\Pi$, we compute the dual graph $\Pi^*_0$ (which is then copied into $\Pi^*$ for modifications).  Step P1 of Procedure 1 asks to compute the reduced combinatorial embedding $\Pi-v$.  This corresponds to removing $v$ from $\Pi$ and a set of planarised edges.  Because at this stage, it is unknown if the embedding $\Pi-v$ will be utilised for the next iteration, it is quicker to instead modify the dual graph $\Pi^*$ to reflect the removal of $v$.  This process is done inside of the procedure \textsc{remove}$(G,\Pi^*,v)$ according to the discussion on contractions in the dual graph in Section \ref{sec:dualgraph}.  Later, if $\Pi-v$ will be utilised for the next iteration, then it is computed, along with the new placement of $v$.

The procedure \textsc{sip}$(G,\Pi^*,v)$ solves the fixed embedding star insertion problem for the vertex $v$ (into $\Pi-v$).  The contractions in $\Pi^*$, discussed in Section \ref{sec:newmethod}, reduce the number of times that shortest paths need to be computed, which is the most costly process of the heuristic.  Then the optimal placement for $v$ is given by \textit{newface}, and shortest paths are computed once more with \textit{newface} as the source vertex.  The list \textit{shortest paths} stores the tree paths from \textit{newface} to each $w \in N_G(v)$.

\subsection{Runtime}

In this subsection we work through the lines of the \textsc{main loop} pseudocode and discuss the time complexity of each operation.  The majority of the work is simple vector manipulation and so some detail is left out here.  As will be seen, in each iteration, the steps performed take no more than $O((k+n)m)$ time.

At lines 5 and 6 we find the faces and dual graph of $\Pi$.  This can be achieved by scanning the edges of $\Pi$ in a clockwise manner and time required for this is $O(k+m)$.

Next, during the loop at line 7, we delete a vertex $v$ and search for a better placement for $v$.  Potentially every vertex may be tried before the algorithm moves on.  So the following procedures may be repeated up to $n$ times per iteration.

In the procedure \textsc{remove}, which is entered at line 9, a number of edge contractions are performed.  In the drawing of $G$ which is associated with the current embedding $\Pi$, let $k_v$ denote the number of crossings on the edges incident to vertex $v$.  Then the time required for the corresponding edge contractions is $O(k_v + d(v))$ for each $v$.  Summing over all $n$ vertices in the aforementioned loop, this becomes a worst case of $O(k+m)$.

In the procedure \textsc{sip}, the contractions at line 4 can be performed in \\
$O(\sum_{w \in N(v)} d(w))$ time, and summing over all vertices, this becomes $O(nm)$.  At line 5 we find shortest paths on an unweighted planar graph (a simple breadth-first search) which can be done in $O(k+n)$ time and this is repeated for each $w \in N_G(v)$ by the loop at line 2.  Then, summing over all vertices, this becomes $O((k+n)m)$.

Back in the \textsc{main loop} the following procedures happen only once an improvement has been found, so only once per iteration.  At line 17 we fix the new placement and update the existing data to reflect the new placement.  Updating the crossing order list and crossing orientation list discussed in Section \ref{sec:datastruct} can be performed in $O(k)$ time.  Updating the 4 clockwise and anticlockwise numbers discussed in Section \ref{sec:datastruct} can be done in $O(k+m)$ time.

Any required subdivisions are checked for at line 18 by scanning the crossings on every edge to check whether it crosses the same edge more than once.  This scan can be performed in $O(k)$ time.  If a subdivision is required then the corresponding lists need to be updated and this also happens in $O(k)$ time.  Note that these subdivisions are a very rare occurance in practice and when they do occur, a check is put in place at each iteration thereafter to see if the subdivision can be undone.  This additional check can be performed in $O(k)$ time.  If a subdivision is required to be undone, the corresponding lists need to be updated and this happens in $O(k+m)$ time.  We remark that any subdivisions do have an effect on the runtime of future iterations because they cause $n$ to grow, and bounding the time increase is difficult.  Because these subdivisions are rare cases which are usually removed swiftly in subsequent iterations, we conclude that for practical purposes the additional runtime is negligible.  We also remark that the total number of iterations is at most the number of crossings in the initial drawing of $G$.  Hence a na\"ive bound on the total runtime is $O((\bar{k}+n)\bar{k}m)$ where $\bar{k}$ is the initial number of crossings.  This emphasises the dependency between the quality of the initial drawing and the overall performance of the heuristic.

\begin{algorithm}
\begin{algorithmic}[1]
\Procedure{MAIN LOOP}{}
\State \textit{current cr} $\gets$ $cr_D(G)$
\While {true}
\State improvement found $\gets$ \textit{false}
\State Find the faces of $\Pi$.
\State $\Pi_0^* \gets$ dual graph of $\Pi$
\For {$v \in V(G) $}
\State $\Pi^* \gets \Pi^*_0$ (make a copy of $\Pi^*_0$)
\State $\Pi^* \gets$ REMOVE$(G,\Pi^*,v)$
\State $(\textit{new cr, newface, shortest paths}) \gets$ SIP$(G,\Pi^*,v) $
\If {$\textit{new cr} < \textit{current cr}$}
\State improvement found $\gets true$
\State \textbf{break}
\EndIf
\EndFor
\If {improvement found}
\State Update $\Pi$ to reflect new placement using \textit{newface} and \textit{shortest}
\hspace*{1.75cm} \textit{paths}.
\State Check if any subdivisions are needed.
\State Check if any previous subdivisions can be removed.
\State $\textit{current cr} \gets \textit{new cr}$
\State \textbf{continue}
\Else
\State \textbf{break}
\EndIf
\EndWhile
\State \textbf{return} $(\textit{current cr,}\Pi)$
\EndProcedure
\caption{Main procedure of the heuristic.  Inputs are a combinatorial embedding $\Pi$ corresponding to some intial drawing $D$ of $G$, which is represented by the data structures discussed in Section \ref{sec:datastruct}. \label{alg1}}
\end{algorithmic}
\end{algorithm}

\begin{algorithm}
\begin{algorithmic}[1]
\Procedure{REMOVE$(G,\Pi^*,v)$}{}
\For{$e^* \in E(\Pi^*)$}
\If{$e^*$ corresponds to an edge of $G$ which is incident to $v$}
\State Contract $e^*$.
\EndIf
\EndFor
\State \textbf{return} $(\Pi^*)$
\EndProcedure
\caption{Vertex deletion procedure.  Given a dual graph $\Pi^*$ and a vertex $v$ of $G$, this performs edge contractions in the dual according to the discussion in Section \ref{sec:dualgraph}. \label{alg2}}
\end{algorithmic}
\end{algorithm}

\begin{algorithm}
\begin{algorithmic}[1]
\Procedure{SIP$(G,\Pi^*,v)$}{}
\For{$w \in N_G(v)$}
\State $\Pi^{**} \gets \Pi^*$ (make a copy of $\Pi^*$)
\State In $\Pi^{**}$, contract the cycle formed by dual edges corresponding to
\hspace*{1.1cm} edges incident to $w$ in $\Pi$, call the contracted vertex $w_d$.
\State $\textrm{dist}_w \gets$ Shortest path algorithm$(\Pi^{**},w_d)$.
\State Set the dist of vertices contracted to form $w_d$ to zero.
\EndFor
\State $\textit{newface} \gets \textrm{argmin}_k(\sum_{w \in N(v)} \textrm{dist}_w(k))$
\State \textit{shortest paths }$\gets$ Shortest path algorithm$(\Pi^*,$\textit{newface}$)$
\State \textbf{return} (\textit{new cr}, \textit{newface}, \textit{shortest paths})
\EndProcedure
\caption{Star insertion problem solver.  Given a dual graph $\Pi^*$ along with a vertex $v$ of $G$, this performs edge contractions in the dual graph according to the discussion in Section \ref{sec:newmethod}.  Then the fixed embedding version of the star insertion problem is solved for $v$. \label{alg3}}
\end{algorithmic}
\end{algorithm}

\newpage
\section{Experiments}
\label{sec:experiments}

\subsection{Experimental setup}

In this section, we consider the performance of our proposed heuristic on various sets of instances. As mentioned previously, we have implemented our heuristic in both C and \textsc{MATLAB}, and here we report on the C implementation, which we call QuickCross.

Each of the experiments reported on here were conducted on a 2.6GHz AMD Opteron 6282 SE with 4GB RAM and running Centos 6.7 OS. In order to compare the various schemes discussed in Section \ref{sec:designmeth}, each experiment is repeated nine times, once for each combination of the three initial embedding schemes ({\em circle, planar, spring}), and the three minimisation schemes ({\em first, best, biggest face (bf)}). Then, for each of these nine parameter settings, we try 100 different random permutations of the vertex labels and record the result with the least number of crossings.  In such a case, we shall say that the graph was {\em run with 100 random permutations}.

We will consider four sets of instances, the first two of which contain sparse graphs, and the latter two of which contain dense graphs. In particular, the sparse instances considered are two of the sets of instances which were used for benchmarking crossing minimisation heuristics in \cite{chim4}, \cite{gut} and \cite{gut:mutz}. They are known respectively as the KnownCR graphs and the Rome graphs. The dense instances considered are sets of complete graphs, and complete bipartite graphs. We now briefly describe the experiments that will be carried out for each of the sets.

\begin{itemize}\item {\bf KnownCR graphs} - these are a set of instances containing between 9 and 250 vertices, first collected by Gutwenger \cite{gut}, which can be further partitioned into four families of graphs as follows:
\begin{itemize}
\item $C_i \times C_j$: the Cartesian product of the cycle on $i$ vertices with the cycle on $j$ vertices.  The instances contain graphs with $3 \leq i \leq 7$ and $j\geq i$ such that $ij \leq 250$.\\

\item $G_i \times P_j$: the cartesian product of the path on $j+1$ vertices with one of the 21 non-isomorphic connected graphs on 5 vertices, where $i$ denotes which of the 21.  The instances contain graphs with $3 \leq j \leq 49$.\\

\item $G_i \times C_j$: the cartesian product of the cycle on $j$ vertices with one of the 21 non-isomorphic connected graphs on 5 vertices, where $i$ denotes which of the 21.  The crossing number of these graphs are only known for some of the $G_i$ and only these cases are included.  The instances contain graphs with $3 \leq j \leq 50$.\\

\item The Generalised Petersen graphs P$(j,2)$ and P$(j,3)$, on $2j$ vertices.  We shall only use those of type P$(j,3)$ as P$(j,2)$ (studied in \cite{g13}) are easy for heuristics to solve as has already been observed in \cite{chim4}.  The instances contain graphs with $9 \leq j \leq 125$.\\
\end{itemize}
Unlike the other sets of instances in this section, all of the crossing numbers for the KnownCR instances are known, and hence we can compare how close QuickCross gets to the correct value for various scheme combinations. In particular, we report on the average relative deviation between the crossing numbers and the values obtained by QuickCross. In order to illustrate the work performed during the main loop of QuickCross, we also report the average relative deviation after only the initial embedding is finished. We also compare the runtimes of the various scheme combinations, separated into the time spent producing the initial embedding, and the time spent in the main loop of the heuristic. Finally, results on solution quality for other crossing minimisation heuristics have been reported in \cite{chim4}, and so we compare our results to theirs.

\item {\bf Rome graphs} - these are a set of 11,528 graphs which have been constructed from real-life applications, first described by Di Battista et al \cite{graphdrawing}. They contain between 10 and 100 vertices, and are very sparse with average edge density of 1.35. The larger graphs in this set have unknown crossing numbers, since they are too large for the current exact methods to solve. Hence, it is impossible to report on how close QuickCross gets to the true crossing number. However, in \cite{gut:mutz} and \cite{gut}, the largest graphs in the Rome graph set were considered, that is, the 140 graphs with exactly 100 vertices. For these graphs, the average numbers of crossings found for various crossing minimisation heuristics were reported. We compare the results of QuickCross to these values, and report on the runtimes for each of the scheme combinations. The runtimes are separated into time spent producing the initial embedding, and time spent in the main loop of the heuristic

%

\item {\bf Complete graphs} - Although the crossing number of the complete graph $K_n$ is not known for for $n \geq 13$, the value is conjectured, and typically assumed to be correct. We compare the nine combinations of schemes to see how close to the conjectured value each of them is able to get, for various sizes of complete graphs up to $n = 50$. We indicate how many crossings are obtained after the initial embedding, as well as at the conclusion of the heuristic. We also provide the runtimes, again separated into time spent producing the initial embedding, and time spent in the main loop of the heuristic.

\item {\bf Complete bipartite graphs} - Much like the complete graphs, the crossing number of the complete bipartite graph $K_{n_1,n_2}$ is only known in general for $n_1 \leq 6$, but the value is conjectured and typically assumed to be correct. Again, we compare the nine combinations of schemes to see how close to the conjectured value they can get for values up to $n_1,n_2 = 40$, and report the same data as for the Complete graphs.

\end{itemize}

It should be noted that there are two other sets of instances which were considered in \cite{chim4}, namely the AT\&T graphs and the ISCA graphs. However, we have chosen not to include them in our experiments for the following reasons. First, the crossing numbers of the instances contained in them are unknown, so we cannot report on how close QuickCross gets to the crossing number. This is also true for the Rome graphs, however for that set there have been experiments reported on in \cite{gut} that list the average number of crossings found, and so we can do a meaningful comparison. In contrast, for the AT\&T and ISCA graphs, the only reported results (e.g. see \cite{chim4}) perform comparisons to the best solutions found by the heuristics involved within that experiment. There is no meaningful way to compare the results of QuickCross to the results in \cite{chim4}, and hence we omit these two sets of instances from consideration here.

\subsection{KnownCR results}\label{sec:knowncr}
We partitioned the graphs into the four families described above and ran each with the 9 possible combinations of schemes.  Each graph was run with 100 random permutations and the minimum found solution was compared to the actual crossing number by computing the {\em percent relative deviation}.  Let $k$ denote the minimum found solution, then the percent relative deviation from $cr(G)$ is: $100(k-cr(G))/cr(G)$.  The average of these numbers was then taken over each of the four families of graphs and these results are displayed in Figure \ref{fig5}, which we now describe in detail.

For each of the nine scheme combinations in Figure \ref{fig5}, there are two bars displayed, specifically a light grey and a dark grey bar. The dark grey bar indicates the average percent relative deviation once the initial embedding is completed (but before the main loop of the heuristic is run), while the light grey bar indicates the average percent relative deviation at the conclusion of the heuristic.  Therefore a large difference between the dark grey and light grey bars represents a large reduction in crossings achieved during the heuristic.  We append the five best reported methods from \cite{chim4} to Figure \ref{fig5} for comparison.  Each of these five methods also utilised 100 random permutations and then chose the minimum found solution.

We observe that for the graphs of type $C_i \times C_j$ and $G_i \times C_j$, the {\em circle} embedding and the {\em planar} embedding perform very well and they outperform the other results by approximately $2.5 \%$, including those from \cite{chim4}.  For the $C_i \times C_j$ and $G_i \times C_j$ graphs, the {\em spring} embedding performs relatively poorly.  On the other hand, for the graphs of type $G_i \times P_j$ and P$(j,3)$, the {\em circle} and {\em planar} embeddings perform poorly and the {\em spring} embedding performs better.  For the P$(j,3)$ graphs the {\em spring} embedding produced average relative deviations which are approximately equal to the best reported results in \cite{chim4}, while they are slightly worse for the $G_i \times P_j$ graphs.  The {\em best} scheme performed worse than {\em first} and {\em bf} under the same initial embedding scheme in almost all cases, with the sole exception of $G_i \times C_j$ and {\em planar} embedding.

Runtimes were analysed by taking an average over the 100 random permutations for each graph. In Figure \ref{fig6} we display the average runtimes to complete the initial embedding, while in Figure \ref{fig7} we display the average runtimes for the remainder of the heuristic.

We observe that, as indicated in Section \ref{sec:init}, the {\em circle} embedding computes an initial embedding the quickest, however it creates many additional crossings (see the dark grey bars in Figure \ref{fig5}), and consequently has a longer heuristic runtime.  Alternatively, the {\em planar} embedding scheme computes an embedding almost as quick and the embedding has far fewer crossings, which results in a significantly lower heuristic runtime.

\begin{figure}[h!]
\begin{centering}
\includegraphics[width=1\linewidth]{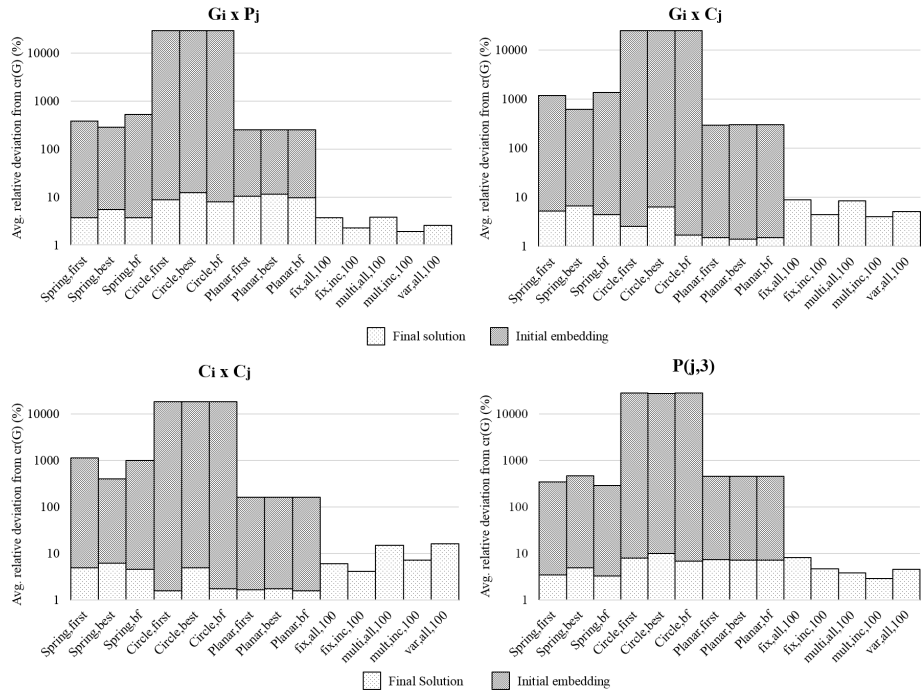}
\caption{Average percent relative deviations from the crossing numbers for four families within the KnownCR graphs. \label{fig5}}
\end{centering}
\end{figure}

\begin{figure}[h!]
\begin{centering}
\includegraphics[width=0.8\linewidth]{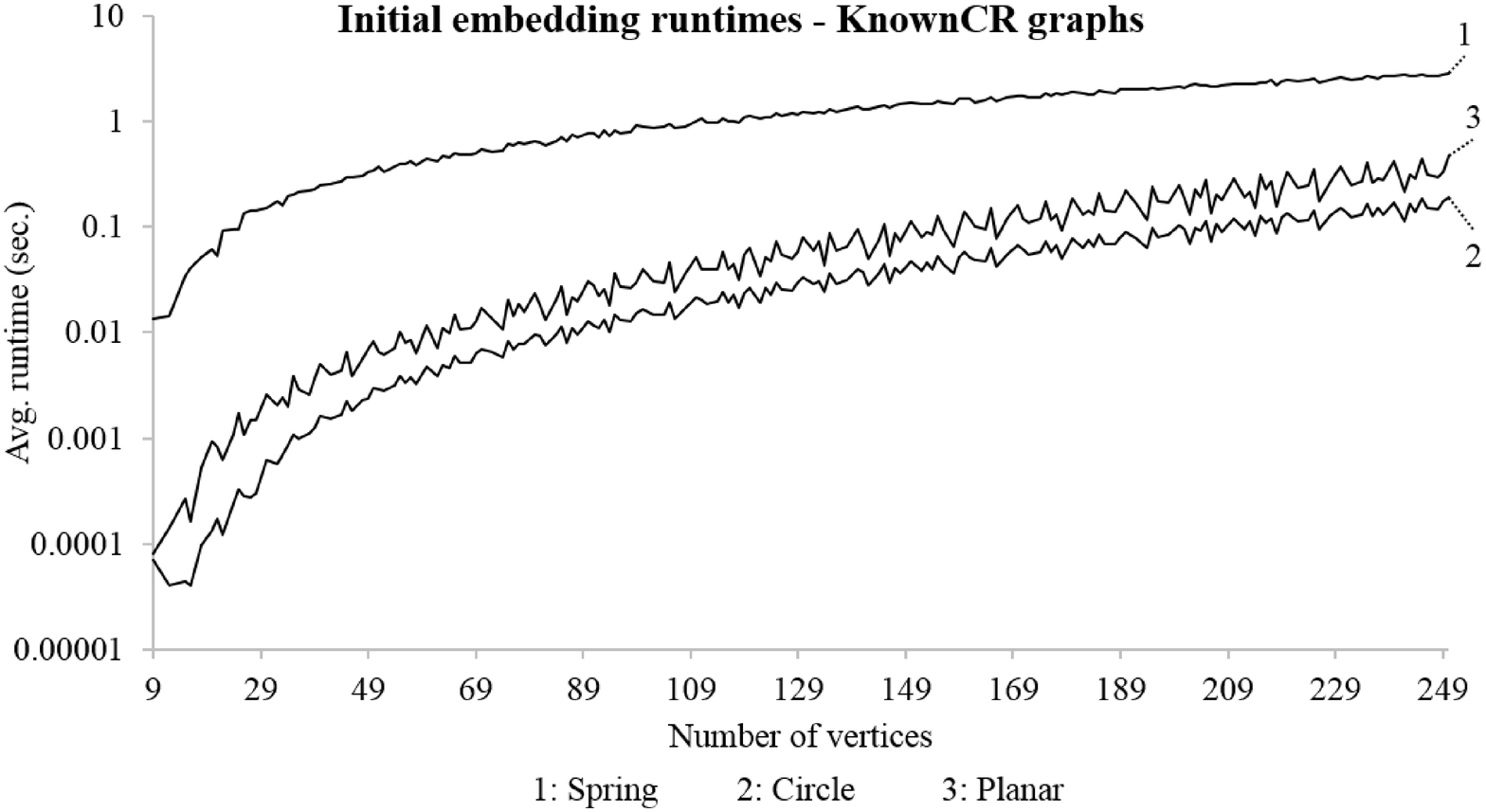}
\caption{Average runtime to produce an initial embedding (sec.) compared to number of vertices for the KnownCR graphs. \label{fig6}}
\end{centering}
\end{figure}

\begin{figure}[h!]
\begin{centering}
\includegraphics[width=0.8\linewidth]{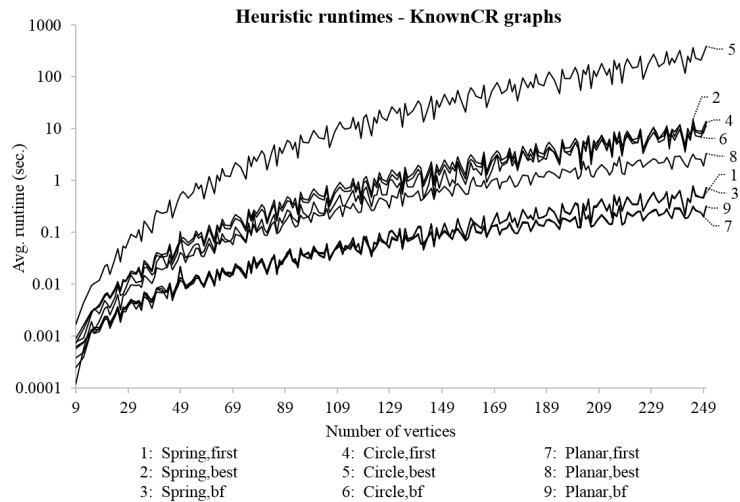}
\caption{Average heuristic runtime per random permutation (sec.) compared to number of vertices for the KnownCR graphs. \label{fig7}}
\end{centering}
\end{figure}

\newpage
\subsection{Rome graphs}
For the 140 graphs on 100 vertices in the Rome graph set, we repeat two experiments that have been previously performed in \cite{gut}.  In the first experiment, for each of the 9 possible combinations of schemes, each graph was run with 100 random permutations. In the second experiment, the number of random permutations is increased to 500.  In each case, we record the smallest number of crossings found for each graph, and report on the average minimum number of crossings over the 140 graphs in Tables \ref{table:1} and \ref{table:2} respectively.  In order to demonstrate the work performed by the main part of the heuristic, we provide the number of crossings at the completion of the initial embeddings as well. We also include the numbers obtained from the previous experiments in \cite{gut} for comparison of solution quality.

In addition, we include average runtimes in Table \ref{table:1}, separated into time spent in the main part of the heuristic, and time spent in the initial embedding. Since the experiments in \cite{gut} are from 2010, it is not meaningful to compare runtimes, and so we report only on the runtimes of QuickCross. Also, since the average runtime after 100 permutations is effectively the same as after 500 permutations, we omit the runtimes from Table \ref{table:2}.

For these experiments, we observe that after 100 random permutations, the {\em planar} embedding scheme outperforms both {\em circle} and {\em spring} in both solution quality and average runtime.  Interestingly, the {\em best} minimisation scheme outperforms the other minimisation schemes under the same embedding scheme in each case.  This result is different to the KnownCR graphs in which the {\em best} scheme was usually the worst performing scheme.  With the exception of {\em npc-var-inc-100}, every configuration compares favourably to the experiments in \cite{gut}. After 500 random permutations, the {\em circle,best} configuration resulted in the smallest average crossings, and each of our nine schemes outperform the experiments in \cite{gut}. We do not include the average runtimes in Table \ref{table:2} as they are almost identical to those in Table \ref{table:1}.

\begin{table}[h!]
\begin{center}
\resizebox{0.9\textwidth}{!}{

\begin{tabular}{ c |c |c|c|c }
\multicolumn{5} {c} {\textbf{Rome graphs - 100 random permutations}} \\
 Method & \begin{tabular}{@{}c@{}}Avg. final\\crossings\end{tabular} & \begin{tabular}{@{}c@{}}Avg. initial\\crossings\end{tabular} & \begin{tabular}{@{}c@{}}Avg. heuristic \\ runtime (sec.)\end{tabular} & \begin{tabular}{@{}c@{}}Avg. embed \\runtime (sec.)\end{tabular}\\
\hline \rule{0pt}{2.6ex}
 planar,best & 25.757 & 54.664 & 0.0368 & 0.0061 \\
 planar,first & 25.779 & 54.664 &  0.0099 & 0.0061\\
 planar,bf & 25.800 & 54.664 &  0.0110 & 0.0061 \\
npc-var-inc-100 & 25.800 & {--} & {--} & {--}\\
 circle,best & 25.829 & 919.87 &  0.2547 & 0.0029 \\
 spring,best & 25.850 & 80.971 &  0.0469 & 0.2591 \\
 circle,bf & 25.886 & 919.87 &  0.0425 & 0.0029 \\
 circle,first & 25.900 & 919.87 &  0.0369 & 0.0029 \\
 spring,first & 25.950 & 80.971 &  0.0105 & 0.2591 \\
 spring,bf & 25.964 & 80.971 &  0.0118 & 0.2591 \\
 npc-fix-inc-100 & 26.600 & {--} &  {--} & {--}\\
 npc-var-all-100 & 27.200 & {--} &  {--} & {--}\\
 npc-fix-all-100 & 28.300 & {--} &  {--} & {--}\\
\end{tabular}}
\caption{The average number of crossings found over 100 permutations and the average heuristic and initial embedding runtime (sec.) per random permutation for the graphs on 100 vertices in the Rome graphs. \label{table:1}}
\end{center}

\end{table}

\begin{table}[h!]
\begin{center}
\resizebox{0.58\textwidth}{!}{
\begin{tabular}{ c |c}
\multicolumn{2} {c} {\textbf{Rome graphs - 500 random permutations}} \\
 Method & {Avg crossings} \\
\hline \rule{0pt}{2.6ex}
 circle,best & 25.157  \\
 planar,best & 25.214  \\
 planar,first & 25.229  \\
 planar,bf & 25.250  \\
 circle,first & 25.300  \\
 spring,best & 25.307 \\
 circle,bf & 25.313  \\
 spring,bf & 25.407  \\
 spring,first & 25.457  \\
 npc-var-inc-500 & 25.510 \\
 npc-fix-inc-500 & 26.090  \\
 npc-var-all-500 & 26.650  \\
 npc-fix-all-500 & 27.500  \\
\end{tabular}}
\end{center}
\caption{The average number of crossings found over 500 permutations for the graphs on 100 vertices in the Rome graphs.  The list is sorted by smallest average number of crossings. \label{table:2}}
\end{table}

Finally, we note that there was an additional experiment conducted on the Rome graphs contained in \cite{chim4}. However, in that experiment, rather than reporting on average numbers of crossings, the results were instead compared to the best known results discovered by the heuristics used in the experiment. Hence, there is no meaningful way to compare the results of QuickCross to these results.

\subsection{Complete graphs}
The crossing number of the complete graph $K_n$ is conjectured (e.g. see Guy \cite{hill}) to be equal to
$$H(n) := 1/4 \left \lfloor n/2 \right \rfloor \left \lfloor (n-1)/2 \right \rfloor \left \lfloor (n-2)/2 \right \rfloor \left \lfloor (n-3)/2 \right \rfloor.$$

Although this conjecture is widely believed to be correct, it has only been confirmed for $n \leq 12$ despite considerable effort to extend the results further \cite{pan}. We ran the graphs $K_n$ for $5 \leq n \leq 50$.  Each graph was run with 100 random permutations and the minimum found solution was compared to $H(n)$ by computing the percent relative deviation from $H(n)$.   These results are displayed in Tables \ref{table:3} for some selected values of $n$, and the runtimes are provided in Table \ref{table:4}.

For these graphs, we observe that when $n$ was odd, every scheme combination was able to obtain a drawing with $H(n)$ crossings. However, when $n$ was even, each scheme reached a value which was usually very close but not equal to $H(n)$. The average runtime under the {\em best} scheme is significantly higher than the other minimisation schemes simply due to the vast amount of additional work required to consider every vertex each iteration.

We now briefly look at the effect of the initial embedding schemes for these instances. In Table \ref{table:5}, we display the percent relative deviation at the conclusion of each of the initial embedding schemes. As can be seen, the {\em planar} initial embedding scheme often provides an embedding for which the number of crossings is very close or even equal to $H(n)$, and hence very little additional work is required by the heuristic. Since the {\em planar} initial embedding scheme is similar in character to the planarisation method, we conclude that heuristics based on the planarisation method would also be effective for these instances. Notably, unlike the sparser KnownCR and Rome instances, for these dense instances the {\em planar} initial embedding scheme takes much longer than the other two initial embedding schemes. This is because the amount of work performed by the {\em circle} and {\em spring} initial embedding schemes depends primarily on the number of vertices, rather than edges. However, the price paid by the {\em circle} and {\em spring} initial embedding schemes schemes is that they result in many more crossings than the {\em planar} scheme, and so the heuristic has to perform much more additional work to descend to a solution; also, the individual iterations (whose time depends on the current number of crossings) take longer as well. Overall, the {\em planar} initial embedding scheme performed the best in terms of both solution quality and total execution time for these instances.

\begin{table}[h!]
\resizebox{\textwidth}{!}{
\begin{tabular}{l|c|c|c|c|c|c|c|c|c}
\multicolumn{10}{c}{  \textbf{Final crossings (\%) for $K_{n}$}}\\
\hline
n & s,first & s,best& s,bf  & c,first & c,best & c,bf  & p,first & p,best& p,bf\\
\hline
20 &0 &0 &0 &0 &0 &0 &0 &0 &0\\
25 &0 &0 &0 &0 &0 &0 &0 &0 &0\\
30 &0.0105 &0.0209& 0.0209  &0.0209 &0&0.0419 & 0.0209 &0.0209 &0.0314 \\
35 &0 &0 &0 &0 &0 &0 &0 &0 &0\\
40 &0.0185 &0.0246& 0.0246  &0.0185 &0.0154 &0.0400 &0.0062&0.0062&0.0092 \\
45 &0 &0 &0 &0 &0 &0 &0 &0 &0\\
50 & 0.0169& 0.0266& 0.0229& 0.0169& 0.0145& 0.0507& 0.0024 &0.0024& 0.0036\\
\end{tabular}}
\caption{Percent relative deviations from $H(n)$ after the conclusion of the heuristic, for the complete graphs $K_n$.  The {\em spring}, {\em circle} and {\em planar} initial embedding schemes are denoted respectively as {\em s},{\em c} and {\em p}. \label{table:3}}
\end{table}

\begin{table}[h!]
\resizebox{\textwidth}{!}{
\begin{tabular}{l|c|c|c|c|c|c|c|c|c}
\multicolumn{10}{c}{  \textbf{Heuristic runtime (sec.) for $K_{n}$}}\\
\hline
n & s,first & s,best & s,bf  & c,first & c,best & c,bf & p,first & p,best& p,bf\\
\hline
20&	0.3978&	2.1947	&0.4572	&0.3978	&2.1816	&0.4592	&0.3721	&0.5217	&0.3872\\
25&	1.4625&	9.8596	&1.5680	&1.4418	&10.746	&1.6956	&0.5380	&0.5489	&0.5805\\
30&	7.0983&	63.244	&7.5867	&6.7814	&60.464	&8.0258	&3.1954	&7.9235	&2.8368\\
35&	15.984&	176.47	&18.201	&15.676	&180.46	&21.244	&5.2559	&4.8622	&5.4484\\
40&	61.769&	794.62	&63.593	&60.615	&749.43	&69.638	&22.096	&72.684	&18.803\\
45&	114.06&	1550.4	&116.29	&101.83	&1215.1	&131.83	&27.644	&28.133	&28.492\\
50&	280.45&	3930.3	&299.73	&272.84	&3582.1	&345.97	&82.517	&252.32	&72.003\\
\end{tabular}}
\caption{Average heuristic runtime (sec.) per random permutation for the complete graphs $K_n$.  The {\em spring}, {\em circle} and {\em planar} initial embedding schemes are denoted respectively as {\em s},{\em c} and {\em p}.  \label{table:4}}
\end{table}

\begin{table}[h!]
\begin{center}
\resizebox{\textwidth}{!}{
\begin{tabular}{l|c|c|c}
\multicolumn{4}{c}{  \textbf{Initial embedding crossings (\%)}}\\
\hline
$n$ & spring & circle & planar\\
\hline
20	&138.21&199.07	&0.3704\\
25	&124.52	&190.40	&0 \\
30	&122.72	&186.81	&0.1151\\
35	&118.41	&183.09	&0\\
40	&114.62	&181.27	&0.0492\\
45	&114.02	&179.22	&0\\
50	&112.72	&178.14 &0.0254\\

\end{tabular}
\quad
\begin{tabular}{l|c|c|c}
\multicolumn{4}{c}{  \textbf{Initial embedding runtime (sec.)}}\\
\hline
$n$ & spring & circle & planar\\
\hline
20	&0.0903&0.0046	&0.0368\\
25	&0.1641	&0.0113	&0.1435\\
30	&0.2378	&0.0260	&0.4690\\
35	&0.3334	&0.0494	&1.2009\\
40	&0.4855	&0.0935	&3.8485\\
45	&0.6783	&0.1441	&7.2532\\
50	&0.9905	&0.2239	&14.1899\\

\end{tabular}}
\caption{Percent relative deviations from $H(n)$ and average runtime (sec.) per random permutation after only the initial embedding for the complete graphs $K_{n}$. \label{table:5}}
\end{center}
\end{table}

\subsection{Complete bipartite graphs}
The crossing number of the complete bipartite graph $K_{n_1,n_2}$ is conjectured (e.g. see Zarankiewicz \cite{Zaran}) to be equal to
$$Z(n_1,n_2) := \left \lfloor n_1/2 \right \rfloor \left \lfloor (n_1-1)/2 \right \rfloor \left \lfloor n_2/2 \right \rfloor \left \lfloor (n_2-1)/2 \right \rfloor.$$

We ran the graphs $K_{n_1,n_2}$ for $5 \leq n_1 \leq n_2 \leq 40$.  Each graph was run with 100 random permutations and the minimum found solution was compared to $Z(n_1,n_2)$. For the sake of space, we only report on the cases where $n_1$ and $n_2$ are multiples of five. As can be seen in Table \ref{table:7}, QuickCross was successful in obtaining the conjectured optimum in all cases and for all scheme combinations, except $K_{30,30}$ and $K_{40,40}$ under the {\em circle, best} combination. We conclude that these graphs are relatively easy for this heuristic to obtain a high-quality solution, and we suspect that this is the case for other heuristic methods as well. However, although the conjectured optimum is easily reached for these graphs, the runtimes in Tables \ref{table:9} and \ref{table:10} are comparable to those for the complete graphs, due to edge density. Again, the {\em planar} initial embedding takes a long time compared to the other initial embedding schemes, but nonetheless results in the shortest overall execution time.

\begin{table}[h!]
\begin{center}
\resizebox{0.85\textwidth}{!}{
\begin{tabular}{l|l|c|c|c|c|c|c|c|c|c}
\multicolumn{11}{c}{  \textbf{Final crossings (\%) for $K_{n_1,n_2}$}}\\
\hline
$n_1$ & $n_2$ & s,first & s,best & s,bf  & c,first & c,best & c,bf & p,first & p,best& p,bf\\
\hline
20 &20 &0	&0	&0	&0	&0	&0	&0	&0	&0\\
20 &25 &0	&0	&0	&0	&0	&0	&0	&0	&0\\
20 &30 &0	&0	&0	&0	&0	&0	&0	&0	&0\\
20 &35 &0	&0	&0	&0	&0	&0	&0	&0	&0\\
20 &40 &0	&0	&0	&0	&0	&0	&0	&0	&0\\
25 &25 &0	&0	&0	&0	&0	&0	&0	&0	&0\\
25 &30 &0	&0	&0	&0	&0	&0	&0	&0	&0\\
25 &35 &0	&0	&0	&0	&0	&0	&0	&0	&0\\
25 &40 &0	&0	&0	&0	&0	&0	&0	&0	&0\\
30 &30 &0	&0	&0	&0	&0.0068	&0	&0	&0	&0\\
30 &35 &0	&0	&0	&0	&0	&0	&0	&0	&0\\
30 &40 &0	&0	&0	&0	&0	&0	&0	&0	&0\\
35&35 &0	&0	&0	&0	&0	&0	&0	&0	&0\\
35 &40 &0	&0	&0	&0	&0	&0	&0	&0	&0\\
40 &40 &0	&0	&0	&0	&0.0104	&0	&0	&0	&0\\
\end{tabular}}
\caption{Percent relative deviations from $Z(n_1,n_2)$ after the conlcusion of the heuristic, for the complete bipartite graphs $K_{n_1,n_2}$.  The {\em spring}, {\em circle} and {\em planar} initial embedding schemes are denoted respectively as {\em s},{\em c} and {\em p}. \label{table:7}}
\end{center}
\end{table}

\begin{table}[h!]
\resizebox{\textwidth}{!}{
\begin{tabular}{l|l|c|c|c|c|c|c|c|c|c}
\multicolumn{11}{c}{  \textbf{Heuristic runtime (sec.) for $K_{n_1,n_2}$}}\\
\hline
$n_1$ & $n_2$ & s,first & s,best & s,bf  & c,first & c,best & c,bf & p,first & p,best& p,bf\\
\hline
20	&20	&2.8400		&24.558	&3.8190		&2.7564		&46.949		&3.1964		&2.2230		&7.7435		&2.1894\\
20	&25	&5.0429		&53.444	&10.408	&5.6112		&68.258		&6.0686		&3.7876		&12.377	&3.8264\\
20	&30	&9.0093		&116.50	&19.750	&9.6452		&127.01		&11.766	&8.0512		&33.853	&7.8455\\
20	&35	&15.346	&192.17	&30.434	&16.704	&216.44		&18.525	&11.674	&38.184	&11.253\\
20	&40	&22.457	&300.00	&54.589	&24.270	&335.33		&32.025	&20.120	&87.243	&20.118\\
25	&25	&12.198	&188.04	&15.685	&12.012	&291.18		&11.593	&5.9796		&7.5129		&5.9845\\
25	&30	&19.335	&338.41	&40.499	&20.860	&439.54 	&22.197	&16.013	&68.515	&16.752\\
25	&35	&32.619	&516.72	&58.314	&36.414	&617.74		&36.929	&18.423	&27.148	&19.012\\
25	&40	&48.460	&878.33	&118.68	&53.492	&892.35		&59.053	&40.087	&182.49	&40.713\\
30	&30	&44.485	&915.05	&73.426	&37.343	&1356.8		&49.762	&34.976	&253.29	&36.584\\
30	&35	&66.802	&1237.9	&142.34	&70.659	&1668.1		&85.998	&53.810	&312.59	&57.174\\
30	&40	&91.535	&1924.7	&272.43	&95.482	&1958.8		&137.42	&93.989	&654.51	&98.410\\
35	&35	&140.10	&2547.6	&199.27	&120.53 &4271.6		&132.97	&63.392	&113.75	&68.421\\
35	&40	&186.39	&3739.3	&386.65	&184.25	&7169.7		&220.25	&157.50	&822.96	&174.53\\
40	&40	&391.16	&6957.5	&683.15	&277.11	&12443	&373.28	&280.24	&2216.0	&295.12\\

\end{tabular}}
\caption{Average heuristic runtime (sec.) per random permutation for the complete bipartite graphs $K_{n_1,n_2}$.  The {\em spring}, {\em circle} and {\em planar} initial embedding schemes are denoted respectively as {\em s},{\em c} and {\em p}.  \label{table:9}}
\end{table}

\begin{table}[h!]
\begin{center}
\resizebox{0.85\textwidth}{!}{
\begin{tabular}{l|l|c|c|c}
\multicolumn{5}{c}{  \textbf{Initial embedding crossings (\%)}}\\
\hline
$n_1$ & $n_2$ & spring & circle & planar\\
\hline
20	&20	&129.40	&191.98	&0.0741\\
20	&25	&115.72	&190.63	&0.0926\\
20	&30	&114.22	&199.44	&0.0688\\
20	&35	&173.59	&191.25	&0.0269\\
20	&40	&112.88	&191.04	&0.0146\\
25	&25	&123.94	&195.86	&0.0000\\
25	&30	&112.99	&183.78	&0.0165\\
25	&35	&111.56	&190.31	&0.0000\\
25	&40	&132.30	&181.84	&0.0548\\
30	&30	&125.14	&190.70	&0.0385\\
30	&35	&111.62	&184.08	&0.0231\\
30	&40	&110.75	&199.05	&0.0727\\
35	&35	&125.54	&192.06	&0.0000\\
35	&40	&115.31	&194.58	&0.0237\\
40	&40	&114.97	&176.00	&0.0506\\

\end{tabular}
\quad
\begin{tabular}{l|l|c|c|c}
\multicolumn{5}{c}{  \textbf{Initial embedding runtime (sec.)}}\\
\hline
$n_1$ & $n_2$ & spring & circle & planar\\
\hline
20	&20	&0.2901	&0.0288	&0.3770\\
20	&25	&0.3244	&0.0415	&0.7642\\
20	&30	&0.3142	&0.0585	&1.4208\\
20	&35	&0.3687	&0.0777	&2.4222\\
20	&40	&0.4910	&0.1025	&3.7615\\
25	&25	&0.3933	&0.0635	&1.5622\\
25	&30	&0.4804	&0.0872	&2.8842\\
25	&35	&0.5414	&0.1215	&4.9249\\
25	&40	&0.7074	&0.1654	&8.0501\\
30	&30	&0.6636	&0.1319	&5.1382\\
30	&35	&0.7924	&0.1857	&9.1530\\
30	&40	&0.9909	&0.2375	&14.562\\
35	&35	&1.1670	&0.2496	&15.514\\
35	&40	&1.4672	&0.3407	&26.831\\
40	&40	&2.0471	&0.4551	&39.911\\

\end{tabular}}
\caption{Percent relative deviations from $Z(n_1,n_2)$ and average runtime (sec.) per random permutation after only the initial embedding for the complete bipartite graphs $K_{n_1,n_2}$.  \label{table:10}}
\end{center}
\end{table}

\section{Conclusion}

We have presented a new heuristic approach to minimising crossings, based on repeatedly solving the star insertion problem. There are a number of parameters and scheme choices that can be utilised and these often result in markedly different performance.

The experiments conducted consistently demonstrate that the {\em planar} initial embedding scheme results in the fastest total execution time for the heuristic, compared to the other initial embedding schemes. This appears to remain true even despite taking considerably longer to complete the initial embedding than the other schemes when the instances are dense. The {\em planar} initial embedding scheme also typically produces a high quality solution, although this depends on the character of the instance considered. In particular, we found that some highly structured sparse instances (for example, the Generalized Petersen graphs $P(j,3)$ considered in Section \ref{sec:knowncr}) responded better to other initial embedding schemes, albeit at a cost to execution time.

Regarding the minimisation schemes, our experiments indicate that the {\em first} minimisation scheme typically provides the best balance between a high quality solution and a fast run-time. For very sparse graphs, the {\em best} minimisation scheme sometimes provides marginally higher quality solutions, but its relatively slow runtime makes it unsuitable for large or very dense graphs. The {\em biggest face} minimisation scheme is the most efficient in the early iterations, but our experiments indicate that the {\em first} minimisation scheme often reaches a locally optimal solution in fewer iterations, and hence it is commonly quicker.

Our experiments indicate that the heuristic performs relatively well on dense graphs, albeit with a slower runtime due to the increased edge density. However, the {\em circle} initial embedding scheme may become impractical for dense graphs since the initial number of crossings is likely to be very large, rendering the early iterations very slow.

Overall, our recommendation for practical use is to rely primarily on the {\em planar, first} setting, and if the highest quality solutions are desired, also consider the {\em circle, first} setting.

For the fixed embedding setting, this work answers the question posed by Chimani and Gutwenger in \cite{chim4} about the performance of a heuristic based upon the star/vertex insertion problem.  It would be interesting to transfer these methods into a variable embedding setting, eliminating the dependence on the initial embedding which has a significant impact on the quality of the solutions.
\section*{Acknowledgements}

We are indebted to the two anonymous referees whose thoughtful and detailed suggestions greatly improved the clarity and quality of this manuscript.



\bibliographystyle{plain}
\bibliography{quickcross_revised_jgaa}{}

\end{document}